\documentclass[english]{article}
\usepackage[T1]{fontenc}
\usepackage[latin9]{inputenc}
\usepackage{calc}
\usepackage{amsmath}
\usepackage{amsthm}
\usepackage{amssymb}
\usepackage{xargs}[2008/03/08]

\makeatletter
\theoremstyle{plain}
\newtheorem{thm}{\protect\theoremname}
\theoremstyle{plain}
\newtheorem{lem}[thm]{\protect\lemmaname}
\theoremstyle{plain}
\newtheorem{prop}[thm]{\protect\propositionname}
\theoremstyle{definition} 
\newtheorem{rem}[thm]{\protect\remarkname}
\theoremstyle{definition}
\newtheorem{example}[thm]{\protect\examplename}
\theoremstyle{plain}
\newtheorem{cor}[thm]{\protect\corollaryname}

\@ifundefined{date}{}{\date{}}

\usepackage{tikz}

\usepackage{pgfplots}

\usepackage{authblk}
\usepackage{blindtext}

\usetikzlibrary{arrows}
\usetikzlibrary{3d}

\usetikzlibrary{automata}

\usetikzlibrary{chains}

\usetikzlibrary{positioning}

\usetikzlibrary{backgrounds}

\usetikzlibrary{calc,decorations.pathreplacing}

\pgfplotsset{compat=1.9}

\usepackage{enumitem}
\setenumerate{ label= (\alph*)}
\setenumerate[2]{ label= (\alph{enumi}\arabic*)} 

\newtheorem{setting}[thm]{Setting}
\newtheorem{question}[thm]{Question}
\newtheorem*{thmA}{Theorem A}
\newtheorem*{thmB}{Theorem B}

\makeatother

\usepackage{babel}
\providecommand{\corollaryname}{Corollary}
\providecommand{\examplename}{Example}
\providecommand{\lemmaname}{Lemma}
\providecommand{\propositionname}{Proposition}
\providecommand{\remarkname}{Remark}
\providecommand{\theoremname}{Theorem}

\begin{document}
\title{TTF classes generated by silting modules}
\author{Alejandro Argud\'in-Monroy, Daniel Bravo, Carlos E. Parra\thanks{The first author was supported by the Project PAPIIT-IN100124 Universidad Nacional Aut\'onoma de M\'exico, the second and third  authors were supported by ANID+FONDECYT/REGULAR+1240253.}}

\maketitle
\global\long\def\Mod{\operatorname{Mod}}%

\global\long\def\Add{\operatorname{Add}}%

\global\long\def\Ext{\operatorname{Ext}_{R}}%

\global\long\def\Mor{\operatorname{Mor}}%

\global\long\def\Proj{\operatorname{Proj}}%

\global\long\def\Coker{\operatorname{Coker}}%

\global\long\def\Ker{\operatorname{Ker}}%

\global\long\def\im{\operatorname{Im}}%

\global\long\def\Hom{\operatorname{Hom}_{R}}%

\global\long\def\Ann{\operatorname{Ann}}%

\global\long\def\Gen{\operatorname{Gen}}%

\global\long\def\Ab{\operatorname{Ab}}%

\global\long\def\c{\mathfrak{c}}%
\global\long\def\t{\mathfrak{t}}%

\newcommandx\triple[3][usedefault, addprefix=\global, 1=\mathcal{C}, 2=\mathcal{T}, 3=\mathcal{F}]{(#1,#2,#3)}%

\newcommandx\suc[5][usedefault, addprefix=\global, 1=N, 2=M, 3=K, 4=, 5=]{#1\overset{#4}{\hookrightarrow}#2\overset{#5}{\twoheadrightarrow}#3}%

\newcommand{\Rmod}{R\text{-}{\rm Mod}}
\newcommand{\Tr}{\rm Tr}

\begin{abstract}

We study the conditions under which a TTF class in a module category over a ring is silting. Using the correspondence between idempotent ideals over a ring and TTF classes in the module category, we focus on finding the necessary and sufficient conditions for $R/I$ to be a silting $R$-module, and hence for the TTF class $\Gen(R/I)$ to be silting, where $I$ is an idempotent two-sided ideal of $R$. In our main result, we show that $R/I$ is a silting module whenever $I$ is the trace of a projective $R$-module. Furthermore, we demonstrate that the converse holds for a broad class of rings, including semiperfect rings.

\end{abstract}

\section{Introduction}
Recall that the HRS-tilting
process constructs, from a torsion class $\mathcal{T}$ in $\Rmod$,
an abelian category $\mathcal{A}_{\mathcal{T}}$ such that the bounded
derived categories of $\mathcal{A}_{\mathcal{T}}$ and $\Rmod$
are equivalent under certain conditions (see \cite{zbMATH00867897}). 
There is a plethora of articles with the goal to establish conditions on the class $\mathcal{T}$ such that $\mathcal{A}_\mathcal{T}$ satisfies some of the conditions introduced by Grothendieck in the hierarchy of abelian categories (see \cite{CGM, zbMATH00867897, MT, zbMATH06669084, PS2, PS3, PS4, PS5, bravo2019t, PV}). Silting torsion classes are exactly those torsion classes $\mathcal{T}$
on which the abelian category $\mathcal{A}_{\mathcal{T}}$ has a projective
generator. 
For more details, the reader is referred to \cite{nicolas2019silting,zbMATH01716644,siltingmodules}
or \cite[Proposition 4.9 and Theorem 5.9]{angeleri2019siltingobjects}. 

In fact, a torsion class $\mathcal{T}$ is silting if there exists
a module $S$ and a projective presentation $P_{-1}\overset{\sigma}{\rightarrow}P_{0}\twoheadrightarrow S$
such that $\Gen(S)=\mathcal{T}$ and the class 
\[
\mathcal{D}_{\sigma}:=\{X\in\Rmod\,|\:\Hom(\sigma,X)\text{ is surjective}\}
\]
coincides with $\mathcal{T}$ (see Section \ref{sec:Preliminaries}
for an explanation of the notation). 
In this case, $S$ is called a silting
module \cite[Definition 3.7]{siltingmodules}. In \cite[Proposition 3.10]{siltingmodules},
it is shown that all silting modules are quasi-tilting finendo, i.e.
the class $\Gen(S)$ is closed by products and contained in $S^{\bot_{1}}$.
However, not all quasi-tilting finendo modules are silting. An example
of this can be seen in \cite[Example 5.4]{silting2017}, where the
authors show a situation where $R/I$ is not silting for an idempotent
two-sided ideal $I$. 

Note that, when one has an idempotent two-sided ideal $I$, then $\mathcal{T}:=\Gen(R/I)$
is a torsion class closed by products and subobjects. That is, $\mathcal{T}$
is a torsion torsion-free class (TTF class for short). Conversely,
every TTF class is of the form $\Gen(R/I)$ with $I$ an idempotent
two-sided ideal. Moreover, in this situation $R/I$ is always a finendo
quasi-tilting module. Therefore, the following question naturally
arises: when is the module $R/I$ a silting module? Our main result
gives a sufficient condition on which $R/I$ is a silting module.

\begin{thmA}\label{thm:main}Let $I$ be an idempotent two-sided
ideal of $R$. If $I$ is the trace of a projective left (resp. right) module,
then $R/I$ is a silting left (resp. right) module. 

\end{thmA}

%

We will also show that the converse of the previous result holds true
under some additional conditions. Namely:

\begin{thmB}\label{prop:main}Let $I$ be an idempotent two-sided
ideal of $R$. If $R/I$ admits a projective cover and $R/I$ is silting,
then $I$ is the trace of a projective module. 

\end{thmB}

We organize the structure of the article as follows. In Section 2
we will see the necessary preliminaries. In particular, we will give
the definition and some properties of silting modules and TTF classes.
In Section 3 we will study the properties that a projective presentation
of $R/I$ must fulfill so that it is a silting module. Finally, in
Section 4 we will prove our main results. As an application, we will
see in Example \ref{exa:semiperfecto,local} that, if the ring $R$
is semiperfect and the Jacobson ideal $J$ is not zero and idempotent,
then $R/J$ is not silting. It is worth mentioning that rings of this
type can be found in \cite[Example 5.4]{silting2017}, \cite[Exercise 10.2]{Passman}, \cite[Example 4.3]{bravo2019t}.

\subsection*{Acknowledgments}

Part of this work was done during in a research stay of the first  author
at the Universidad Austral de Chile in 2024. The first author
would like to thank to the Instituto de Ciencias F\'isicas y Matem\'aticas
of the Universidad Austral de Chile for their hospitality and support. 

\section{\label{sec:Preliminaries}Preliminaries}
In this section, we  review the foundational concepts and notations essential for the development of our results. These preliminaries serve as the groundwork for understanding the structure of silting modules and torsion torsion-free classes, as well as other relevant concepts.
\subsection{Modules and maps.} 
Throughout this paper, we will let $R$ be an associative ring
with unit. We will denote the category of left $R$-modules as
$\Rmod$. 
However, all results in this work also hold true for right
$R$-modules. Recall that $\Rmod$ is an abelian category. In particular,
$\Rmod$ is an additive category that admits finite biproducts.
This allows us to visualize morphisms between finite coproducts of
modules as matrices as follows. 
Given two families of modules $\{A_{i}\}_{i=1}^{n}$
and $\{B_{i}\}_{i=1}^{m}$ consider the canonical inclusions $\mu_{i}^{A}:A_{i}\rightarrow\bigoplus_{k=1}^{n}A_{k}$
and the natural projections $\pi_{j}^{B}:\bigoplus_{k=1}^{m}B_{k}\rightarrow B_{j}$.
We can express a morphism $f:\bigoplus_{i=1}^{n}A_{i}\rightarrow\bigoplus_{i=1}^{m}B_{i}$
as the matrix 
\[
f=\left[\begin{smallmatrix}f_{11} & f_{12} & \cdots & f_{1n}\\
f_{21} & f_{22} & \cdots & f_{2n}\\
\vdots & \vdots & \ddots & \vdots\\
f_{m1} & f_{m2} & \cdots & f_{mn}
\end{smallmatrix}\right]_{m\times n}
\]
where $f_{ji}:=\pi_{j}^{B}\circ f\circ\mu_{i}^{A}$. Note that, by
using this visualization, the composition of morphisms can be seen
as the usual product of matrices. 
\subsection{$\text{Hom}$, $\text{Ext}$ and Orthogonal classes.} 
Recall that, for $A,B\in\Rmod$, then $\Ext^{1}(A,B)$ denotes the collection
of equivalence classes of the short exact sequences of the form $\suc[B][E][A]$.
Note that, by using the Baer sum, we have a functor $\Ext^{1}(-,-):\Rmod^{op}\times\Rmod\rightarrow\Mod(\mathbb{Z})$.
The reader is referred to \cite[Chapter VII]{mitchell} or \cite[Section 2.1]{K} for more details. 

We will be considering the following classes of modules associated
to a class of modules $\mathcal{X}$ and a module $S$: 
\begin{align*}
\mathcal{X}^{\bot_{0}} & :=\{M\in\Rmod\,|\:\Hom(X,M)=0\text{, for all }X\in\mathcal{X}\}\\
^{\bot_{0}}\mathcal{X} & :=\{M\in\Rmod\,|\:\Hom(M,X)=0\text{, for all }X\in\mathcal{X}\}\\
\mathcal{X}^{\bot_{1}} & :=\{M\in\Rmod\,|\:\Ext^{1}(X,M)=0\text{, for all }X\in\mathcal{X}\}\\
^{\bot_{1}}\mathcal{X} & :=\{M\in\Rmod\,|\:\Ext^{1}(M,X)=0\text{, for all }X\in\mathcal{X}\}\\
\Gen(S) & :=\{M\in\Rmod\,|\:\text{there is an epimorphism }S^{(\alpha)}\rightarrow M\text{, for some set }\alpha\}.
\end{align*}
When the class $\mathcal{X}$ is a singleton set $\mathcal{X}=\{M\}$,
we will write $M^{\bot_{0}}$ (resp. $^{\bot_{0}}M$, $M^{\bot_{1}},$$^{\bot_{1}}M$)
instead of $\mathcal{X}^{\bot_{0}}$ (resp. ${}^{\bot_{0}}\mathcal{X}$,
$\mathcal{X}^{\bot_{1}}$, $^{\bot_{1}}\mathcal{X}$). 

\subsection{Projective covers}

Let $M\in\Rmod$. A \emph{projective cover} of $M$ is an epimorphism
$\pi:P\rightarrow M$ with $P\in\Proj(R)$ such that, any morphism
$f:P\rightarrow P$ such that $\pi\circ f=\pi$ is an isomorphism
(see  \cite[Definition V.5.5]{GT}). 
It seems to be folklore that the following result holds.
\begin{lem} 
\label{lem:cubierta proyectiva}Let $p:Q\rightarrow M$ be a projective
cover and $p':Q'\rightarrow M$ be an epimorphism with $Q'\in\Proj(R)$.
Then, there is $P\in\Proj(R)$ such that the following statements
hold true:
\begin{enumerate}
\item $Q'\cong Q\oplus P$, 
\item $\Ker(p')\cong\Ker(p)\oplus P$, 
\item and there is a commutative diagram with exact rows as the one below. 
\end{enumerate}
\end{lem}
\[
\begin{tikzpicture}[-,>=to,shorten >=1pt,auto,node distance=2cm,main node/.style=,x=2cm,y=-2cm]

\node (1) at (1,0) {$K\oplus P$};
\node (2) at (2,0) {$Q \oplus P$};
\node (3) at (3,0) {$M$};

\node (1') at (1,1) {$K$};
\node (2') at (2,1) {$Q$};
\node (3') at (3,1) {$M$};

\draw[right hook -> , thin]  (1)  to  node  {$\left[\begin{smallmatrix}k & 0\\0 & 1\end{smallmatrix}\right]$} (2);
\draw[right hook -> , thin, below]  (1')  to  node  {$ k$} (2');

\draw[->> , thin]  (2)  to  node  {$\left[\begin{smallmatrix}p & 0\end{smallmatrix}\right]$} (3);
\draw[->> , thin, below]  (2')  to  node  {$p$} (3');

\draw[ ->> , thin]  (1)  to  node  {$\left[\begin{smallmatrix}1 & 0\end{smallmatrix}\right]$} (1');
\draw[->> , thin]  (2)  to  node  {$\left[\begin{smallmatrix}1 & 0\end{smallmatrix}\right]$} (2');
\draw[- , double]  (3)  to  node  {$$} (3');

\end{tikzpicture}
\]%

\subsection{Torsion pairs}

Let $\mathfrak{x}=(\mathcal{X},\mathcal{Y})$ be a pair of classes
of modules. We say that $\mathfrak{x}$ is a \emph{torsion pair} in
$\Rmod$ if $\mathcal{X}={}^{\bot_{0}}\mathcal{Y}$ and $\mathcal{Y}=\mathcal{X}^{\bot_{0}}$.
In such case, $\mathcal{X}$ is called a torsion class and $\mathcal{Y}$
a torsion-free class. 

If $\mathfrak{x}$ is a torsion pair, then it can be seen that, for
all $M\in\Rmod$, there is an exact sequence $\suc[\mathfrak{x}M][M][(1:\mathfrak{x})M][\epsilon_{M}][\rho_{M}]$
with $\mathfrak{x}M\in\mathcal{X}$ and $(1:\mathfrak{x})M\in\mathcal{Y}$.
Moreover, the induced functor $\mathfrak{x}:\Rmod\rightarrow\mathcal{X}$
is right adjoint to the inclusion functor $\mathcal{X}\rightarrow\Rmod$,
and the induced functor $(1:\mathfrak{x}):\Rmod\rightarrow\mathcal{Y}$
is left adjoint to the inclusion functor $\mathcal{Y}\rightarrow\Rmod$.
Note that the composition $\Rmod\overset{\mathfrak{x}}{\rightarrow}\mathcal{X}\hookrightarrow\Rmod$,
which we also denote as $\mathfrak{x}$, is called the \emph{torsion
radical}. Similarly, we also denote as $(1:\mathfrak{x})$ the composition
$\Rmod\overset{(1:\mathfrak{x})}{\rightarrow}\mathcal{F}\hookrightarrow\Rmod$,
which is called the \emph{torsion coradical}. 

A class of modules is a torsion class if and only if
it is closed under extensions, quotients and coproducts; and 
it is a torsion-free class if and only if it is closed under extensions,
subobjects and products (see \cite[Chapter VI]{ringsofQuotients}
for more details). 

Lastly, observe that given a morphism in $\Rmod$ of the form 
\[
\left[\begin{smallmatrix}a & b\\
c & d
\end{smallmatrix}\right]:X\oplus Y\rightarrow V\oplus W,
\]
we have that 
\[
\mathfrak{x}\left[\begin{smallmatrix}a & b\\
c & d
\end{smallmatrix}\right]=\left[\begin{smallmatrix}\mathfrak{x}a & \mathfrak{x}b\\
\mathfrak{x}c & \mathfrak{x}d
\end{smallmatrix}\right]\text{ and that }(1:\mathfrak{x})\left[\begin{smallmatrix}a & b\\
c & d
\end{smallmatrix}\right]=\left[\begin{smallmatrix}(1:\mathfrak{x})a & (1:\mathfrak{x})b\\
(1:\mathfrak{x})c & (1:\mathfrak{x})d
\end{smallmatrix}\right].
\]
\subsection{Silting modules}

For a morphism $\alpha$ in $\Rmod$, define 
\[
\mathcal{D}_{\alpha}:=\{X\in\Rmod\,|\:\Hom(\alpha,X)\text{ is surjective}\}.
\]
Let $\sigma$ be a map of projective modules. Observe that $\mathcal{D}_{\sigma}\subseteq\Coker(\sigma){}^{\bot_{1}}$
and that $\mathcal{D}_{\sigma}$ is closed under quotients, extensions,
and products (see \cite[Lemma 3.6]{siltingmodules}). For $S:=\Coker(\sigma)$,
we say that $S$ is \emph{silting with respect to $\sigma$} if $\mathcal{D}_{\sigma}=\Gen(S)$.
In this case, observe that $\mathcal{D}_{\sigma}$ is also closed
under coproducts and therefore it is a torsion class in $\Rmod$. 

The following lemma will be necessary. 
\begin{lem}
\label{lemita d sigma}\cite[Lemma 2.1]{silting2017} Let $\sigma\in\Mor(\Proj(R))$
and $\sigma=\iota_{\sigma}\circ\pi_{\sigma}$ be the canonical factorization
through $\im(\sigma)$. Then, $\mathcal{D}_{\sigma}=\mathcal{D}_{\pi_{\sigma}}\cap(\Coker(\sigma)^{\bot_{1}})$. 
\end{lem}

\subsection{TTF triples}

We recall that a \emph{torsion torsion-free triple} (\emph{TTF} triple for
short) in $\Rmod$ is a triple of classes of modules $\triple$
such that $\c=(\mathcal{C},\mathcal{T})$ and $\t=(\mathcal{\mathcal{T}},\mathcal{F})$
are torsion pairs. It is a known fact that TTF triples in $\Rmod$
are in bijection with idempotent two-sided ideals via the maps $\triple\mapsto\c(R)$
and \renewcommandx\triple[3][usedefault, addprefix=\global, 1=\mathcal{C}_{I}, 2=\mathcal{T}_{I}, 3=\mathcal{F}_{I}]{(#1,#2,#3)}%
 $I\mapsto\triple$ where 
\begin{alignat*}{1}
\mathcal{C}_{I} & =\Gen(I)=\{M\in\Rmod\,|\:I\cdot M=M\},\\
\mathcal{T}_{I} & =\Gen(R/I)=\{M\in\Rmod\,|\:I\cdot M=0\},
\end{alignat*}
 and $\mathcal{F}_{I}=(R/I)^{\bot_{0}}$\textbf{ }(see \cite[Corollary 2.2]{jans1965some}).
From now on, $I$ will denote an idempotent two-sided ideal and $\triple$
will be the corresponding TTF triple, with $\c=(\mathcal{C}_{I},\mathcal{T}_{I})$
and $\t=(\mathcal{\mathcal{T}}_{I},\mathcal{F}_{I})$.

\subsection{Recollements}

A \emph{recollement} of $\Rmod$ by abelian categories $\mathcal{A}$
and $\mathcal{B}$ is a diagram of additive functors: 
\[
\begin{tikzpicture}[-,>=to,shorten >=1pt,auto,node distance=2cm,main node/.style=,x=2cm,y=-2cm]

\node (1) at (1,0) {$\mathcal{A}$};
\node (2) at (2,0) {$\Rmod$};
\node (3) at (3,0) {$\mathcal{B}$};

\draw[-> , thin]  (1)  to  node  {$i_*$} (2);
\draw[-> , thin]  (2)  to  node  {$j^*$} (3);

\draw[-> , thin,in=45,out=135, above]  (3)  to  node  {$j_!$} (2);
\draw[-> , thin,in=45,out=135, above]  (2)  to  node  {$i^*$} (1);

\draw[-> , thin,in=-45,out=-135, below]  (3)  to  node  {$j_*$} (2);
\draw[-> , thin,in=-45,out=-135, below]  (2)  to  node  {$i^!$} (1);

\end{tikzpicture}
\]%
such that the following statements hold true: 
\begin{enumerate}
\item the pairs $(i^{*},i_{*})$, $(i_{*},i^{!})$, $(j_{!},j^{*})$, and
$(j^{*},j_{*})$ are adjoint, 
\item the functors $i_{*}$, $j_{!}$, and $j_{*}$ are fully faithful,
and 
\item $\im(i_{*})=\Ker(j^{*})$. 
\end{enumerate}
It has been shown recentely that there is a bijection between equivalence classes
of recollements of $\Rmod$ and TTF triples in $\Rmod$ (see \cite{psaroudakis2014recollements}).
Specifically, a recollement as the one above induces the TTF triple
\[\triple[\mathcal{C}][\mathcal{T}][\mathcal{F}]=\triple[\Ker(i^{*})][\im(i_{*})][\Ker(i^{!})].\] Reciprocally,
a TTF triple $\triple$, with $\c=(\mathcal{C}_{I},\mathcal{T}_{I})$
and $\t=(\mathcal{T}_{I},\mathcal{F}_{I})$, induces the following recollement

\[
\begin{tikzpicture}[-,>=to,shorten >=1pt,auto,node distance=2cm,main node/.style=,x=2cm,y=-2cm]

\node (1) at (1,0) {$\mathcal{T} _I$};
\node (2) at (2,0) {$ \Rmod$};
\node (3) at (3,0) {$ \Mod (R) / \mathcal{T} _I$};

\draw[-> , thin]  (1)  to  node  {$i_*$} (2);
\draw[-> , thin]  (2)  to  node  {$j^*$} (3);

\draw[-> , thin,in=45,out=135, above]  (3)  to  node  {$j_!$} (2);
\draw[-> , thin,in=45,out=135, above]  (2)  to  node  {$i^*$} (1);

\draw[-> , thin,in=-45,out=-135, below]  (3)  to  node  {$j_*$} (2);
\draw[-> , thin,in=-45,out=-135, below]  (2)  to  node  {$i^!$} (1);

\end{tikzpicture}
\]%
where $i_{*}$ is the inclusion, $i_{*}\circ i^{*}=(1:\c)$, and $i_{*}\circ i^{!}=\t$. 

Finally, we deduce the following properties from \cite{psaroudakis2014recollements}.

\begin{prop}
\label{prop:recollement}Consider a recollement in $\Rmod$ as the
one above. Then, the following statements hold true. 
\begin{enumerate}
\item Every $M\in\Rmod$ admits long exact sequences of the form 
\begin{alignat*}{1}
\mathbf{(A)}\qquad & \suc[i_{*}T][j_{!}j^{*}M\rightarrow M][i_{*}i^{*}M]\text{ and }\\
\mathbf{(B)}\qquad & \suc[i_{*}i^{!}M][M\rightarrow j_{*}j^{*}M][i_{*}T']
\end{alignat*}
for some $T,T'\in\mathcal{T}_{I}$.
\item Every $C\in\mathcal{C}_{I}$ admits a short exact sequence $\suc[T][j_{!}j^{*}C][C]$
with $T\in\mathcal{T}_{I}$. 
\item Every $F\in\mathcal{F}_{I}$ admits a short exact sequence $\suc[F][j_{*}j^{*}F][T]$
with $T\in\mathcal{T}_{I}$. 
\item $\im(j_{!})=\mathcal{C}_{I}\cap{}^{\bot_{_{1}}}\mathcal{T}_{I}$. 
\item $\im(j_{*})=\mathcal{F}_{I}\cap\mathcal{T}_{I}^{\bot_{1}}$. 
\end{enumerate}
\end{prop}

\begin{proof} For {(a)} see \cite[Proposition 2.8]{psaroudakis2014recollements}. 
Next, {(b)} follows from the exact
sequence (A) in (a) since $i_{*}i^{*}C=(1:\c)C=0$.
Now {(c)} follows similar arguments as {(b)}. Finally for {(d)} and {(e)}, see \cite[Corollary 4.5]{psaroudakis2014recollements}. 

\end{proof}

\section{Silting TTF triples}

Throughout this section we fix the following setting.
\begin{setting}\label{setting} Given  an idempotent two-sided ideal $I$, consider the following conditions:
\begin{enumerate}
    \item $\triple$ be the corresponding TTF triple, 
    \item $\c:=(\mathcal{C}_{I},\mathcal{T}_{I})$,
    \item $\t:=(\mathcal{\mathcal{T}}_{I},\mathcal{F}_{I})$, and
    \item $S$ be a module with a projective presentation $P_{-1}\overset{\sigma}{\rightarrow}P_{0}\twoheadrightarrow S$ such that $\Gen(S)=\mathcal{T}_{I}\subseteq S^{\bot_{1}}$ and $\sigma=\iota_{\sigma}\circ\pi_{\sigma}$, where $\pi_{\sigma}:P_{-1}\rightarrow\im(\sigma)$ and $\iota_{\sigma}:\im(\sigma)\rightarrow P_{0}$ are the canonical morphisms.
\end{enumerate}

\end{setting}
\begin{prop}
\label{prop:1} With the conditions described in Setting \ref{setting}, $S$ is silting with
respect to $\sigma$ 
if, and only if, the following statements hold
true.
\begin{enumerate}
\item \textup{$(1:\c)\sigma:(1:\c)P_{-1}\rightarrow(1:\c)P_{0}$} is a (split)
monomorphism.
\item $(1:\c)\pi_{\sigma}:(1:\c)P_{-1}\rightarrow(1:\c)\im\sigma$ is an
isomorphism.
\item If $\Hom(\pi_{\sigma},j_{*}D):\Hom(\im\sigma,j_{*}D)\rightarrow\Hom(P_{-1},j_{*}D)$
is an isomorphism, then $D=0$. 
\end{enumerate}
\end{prop}

\begin{proof}
From the exact sequence of functors $\c\overset{\epsilon}{\hookrightarrow}1\overset{\rho}{\twoheadrightarrow}(1:\c)$
and the Snake Lemma, we get the following commutative diagram with
exact rows and columns.

\[
\begin{tikzpicture}[-,>=to,shorten >=1pt,auto,node distance=2cm,main node/.style=,x=2cm,y=-2cm]

\node (1) at (1,0) {$\c P_{-1}$};
\node (2) at (2,0) {$ P_{-1}$};
\node (3) at (3,0) {$(1:\c) P_{-1}$};

\node (1') at (1,1) {$\c P_{0}$};
\node (2') at (2,1) {$ P_{0}$};
\node (3') at (3,1) {$(1:\c) P_{0}$};

\node (1'') at (1,2) {$$};
\node (2'') at (2,2) {$S$};
\node (3'') at (3,2) {$S$};

\draw[right hook -> , thin]  (1)  to  node  {$\epsilon _{P_{-1}}$} (2);
\draw[right hook -> , thin]  (1')  to  node  {$\epsilon _{P_0} $} (2');

\draw[->> , thin]  (2)  to  node  {$\rho_{P_{-1}}$} (3);
\draw[->> , thin]  (2')  to  node  {$\rho_{P_0}$} (3');
\draw[- , double]  (2'')  to  node  {$$} (3'');

\draw[ ->> , thin]  (1)  to  node  {$\c \sigma$} (1');
\draw[-> , thin]  (2)  to  node  {$\sigma$} (2');
\draw[-> , thin]  (3)  to  node  {$(1:\c )\sigma$} (3');

\draw[->> , thin]  (2')  to  node  {$c$} (2'');
\draw[->> , thin]  (3')  to  node  {$(1:\c )c$} (3'');

\end{tikzpicture}
\]%
\noindent $(\implies)$. Assume that $S$ is silting with respect to $\sigma$. 
\begin{enumerate}
\item Using the fact that $(1:\c)P_{-1}\in\mathcal{T}_{I}=\Gen(S)=\mathcal{D}_{\sigma}$,
we get a morphism $\hat{\sigma}:P_{0}\rightarrow(1:\c)P_{-1}$ such
that $\hat{\sigma}\circ\sigma=\rho_{P_{-1}}$. Note that $\hat{\sigma}\circ\epsilon_{P_{0}}=0$.
Thus, there is a morphism $\alpha:(1:\c)P_{0}\rightarrow(1:\c)P_{-1}$
such that $\alpha\circ\rho_{P_{0}}=\hat{\sigma}$. Therefore, (a)
is satisfied because 
\[
\rho_{P_{-1}}=\hat{\sigma}\circ\sigma=\alpha\circ\rho_{_{P_{0}}}\circ\sigma=\alpha\circ(1:\c)\sigma\circ\rho_{P_{-1}}
\]
and $\rho_{P_{-1}}$ is an epimorphism. 
\item Consider the exact sequence $\suc[\Ker(\sigma)][P_{-1}][\im(\sigma)][k][\pi_{\sigma}]$
and the following commutative diagram with exact rows.

\[
\begin{tikzpicture}[-,>=to,shorten >=1pt,auto,node distance=2cm,main node/.style=,x=2cm,y=-2cm]

\node (1) at (1,0) {$\c P_{-1}$};
\node (2) at (2,0) {$ P_{-1}$};
\node (3) at (3,0) {$(1:\c) P_{-1}$};

\node (1') at (1,1) {$\c \im (\sigma)$};
\node (2') at (2,1) {$ \im (\sigma)$};
\node (3') at (3,1) {$(1:\c) \im (\sigma)$};

\draw[right hook -> , thin]  (1)  to  node  {$\epsilon _{P_{-1}}$} (2);
\draw[right hook -> , thin]  (1')  to  node  {$ u$} (2');

\draw[->> , thin]  (2)  to  node  {$\rho_{P_{-1}}$} (3);
\draw[->> , thin]  (2')  to  node  {$p$} (3');

\draw[ ->> , thin]  (1)  to  node  {$\c \pi_{\sigma}$} (1');
\draw[->> , thin]  (2)  to  node  {$\pi_{\sigma}$} (2');
\draw[->> , thin]  (3)  to  node  {$(1:\c )\pi_{\sigma}$} (3');

\end{tikzpicture}
\]%
Here, $\c\pi_{\sigma}$ is an epimorphism because, by the Snake Lemma,
there is an epimorphism from $\Ker((1:\c)\pi_{\sigma})\in\mathcal{T}_{I}$
into $\Coker(\c\pi_{\sigma})\in\mathcal{C}_{I}$, and hence $\Coker(\c\pi_{\sigma})\in\mathcal{C}_{I}\cap\mathcal{T}_{I}=0$.
Now, observe that $\Hom(\pi_{\sigma},(1:\c)P_{-1})$ is surjective
because $(1:\c)P_{-1}\in\mathcal{T}_{I}=\Gen(S)=\mathcal{D}_{\sigma}\subseteq\mathcal{D}_{\pi_{\sigma}}$
by Lemma \ref{lemita d sigma}. Hence, there is a morphism $\epsilon:\im(\sigma)\rightarrow(1:\c)P_{-1}$
such that $\epsilon\circ\pi_{\sigma}=\rho_{P_{-1}}$. Note that $\epsilon\circ u=0$.
Hence, there is a morphism $\epsilon':(1:\c)\im(\sigma)\rightarrow(1:\c)P_{-1}$
such that $\epsilon'\circ p=\epsilon$. Lastly, note that $\epsilon'\circ(1:\c)\pi_{\sigma}=1$
because
\[
\rho_{P_{-1}}=\epsilon\circ\pi_{\sigma}=\epsilon'\circ p\circ\pi_{\sigma}=\epsilon'\circ(1:\c)\pi_{\sigma}\circ\rho_{P_{-1}}.
\]
 Therefore, $(1:\c)\pi_{\sigma}$ is an isomorphism.
\item Consider $D\in\Rmod/\mathcal{T}_{I}$ such that $\Hom(\pi_{\sigma},j_{*}D)$
is an isomorphism. By definition, we have that $j_{*}D\in\mathcal{D}_{\pi_{\sigma}}$.
Hence, $j_{*}D\in\mathcal{D}_{\sigma}=\mathcal{D}_{\pi_{\sigma}}\cap S^{\bot_{1}}=\mathcal{T}_{I}$
since $j_{*}D\in\mathcal{T}_{I}^{\bot_{1}}$ by Proposition \ref{prop:recollement}(e).
It follows that $j_{*}D\in\mathcal{T}_{I}\cap\mathcal{F}_{I}=0$.
We deduce that $D=0$ since $j_{*}$ is faithful. 
\end{enumerate}

\noindent $(\impliedby)$. Assume that (a), (b) and (c) are satisfied. It is
enough to prove that $\mathcal{D}_{\sigma}=\mathcal{T}_{I}$. Let
$T\in\mathcal{T}_{I}$ and consider a morphism $f:P_{-1}\rightarrow T$.
Observe that $\Hom(\rho_{P_{-1}},T)$ is an isomorphism since $\Ker(\rho_{P_{-1}})=\c P_{-1}\in\mathcal{C}_{I}$.
And thus, there is a morphism $\hat{f}:(1:\c)P_{-1}\rightarrow T$
such that $\hat{f}\circ\rho_{P_{-1}}=f$. Recall that, by (b), $(1:\c)\pi_{\sigma}$
is an isomorphism. Now, using the fact that $T\in S^{\bot_{1}}$,
we obtain a morphism $g:P_{0}\rightarrow T$ such that $g \circ \iota_{\sigma}=\hat{f} \circ ((1:\c)\pi_{\sigma})^{-1} \circ \rho_{\im (\sigma)}$.
From the following composition of morphisms, we deduce that $T\in\mathcal{D}_{\sigma}$
and, therefore, $\mathcal{T}_{I}\subseteq\mathcal{D}_{\sigma}.$
\begin{align*}
g\circ\sigma=g\circ\iota_{\sigma}\circ\pi_{\sigma} & =(\hat{f}\circ((1:\c)\pi_{\sigma})^{-1}\circ\rho_{\im(\sigma)})\circ\pi_{\sigma}\\
 & =\hat{f}\circ((1:\c)\pi_{\sigma})^{-1}\circ(1:\c)\pi_{\sigma}\circ\rho_{P_{-1}}\\
 & =\hat{f}\circ\rho_{P_{-1}}=f.
\end{align*}
On the other hand, let $M\in\mathcal{D}_{\sigma}$. Consider $W:=(1:\t)M\in\mathcal{F}$.
We claim that $W\in\mathcal{T}_{I}^{\bot_{1}}$. Indeed, since $\mathcal{T}_{I}=\Gen(S)$,
for every $T\in\mathcal{T}_{I}$ there is a short exact sequence $\suc[T'][S^{(\alpha)}][T]$
with $T'\in\mathcal{T}_{I}$. Then, since $W\in\mathcal{D}_{\sigma}$
and $\mathcal{D}_{\sigma}\subseteq S^{\bot_{1}}$, we have that the
following exact sequence 
\[
0=\Hom(T',W)\rightarrow\Ext^{1}(T,W)\rightarrow\Ext^{1}(S^{(\alpha)},W)=0.
\]
Therefore, $W\in\mathcal{F}\cap\mathcal{T}_{I}^{\bot_{1}}=\im(j_{*})$.
And thus, $W=j_{*}D$ for some $D\in\Rmod/\mathcal{T}_{I}$. Now,
since $W\in\mathcal{D}_{\sigma}\subseteq\mathcal{D}_{\pi_{\sigma}}$,
$\Hom(\pi_{\sigma},W)$ is surjective. And thus, by (c), we have that
$W=0$. This means that $M=\t M$. And hence, $\mathcal{D}_{\sigma}=\mathcal{T}_{I}$. 

\end{proof}
\begin{rem}
\label{rem:presentacion diagrama} With the consideration of Setting \ref{setting}. Statements (a) and (b) in Proposition \ref{prop:1} are satisfied if and only if we have the commutative diagram below with exact columns and rows, where $C,C'\in\mathcal{C}_{I}$ and $T,T'\in\mathcal{T}_{I}$:
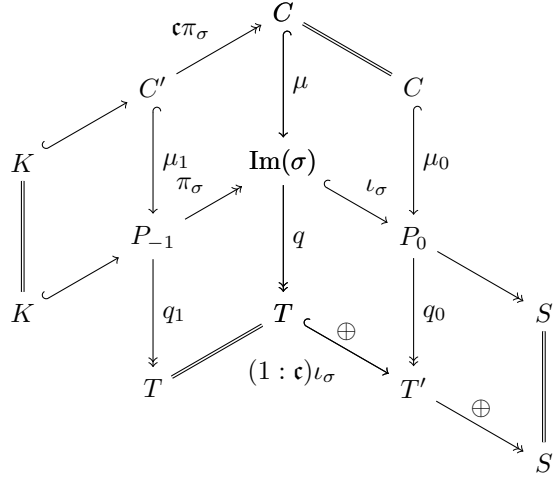
\[
\begin{tikzpicture}[-,>=to,shorten >=1pt,auto,node distance=2cm,main node/.style=,x=2cm,y=-2cm]

\begin{scope}
   \node (3)  at (0,0)         {$T$};
   \node (3')  at (-30:2cm)     {$T'$};
   \node (3'')  at (-30:4cm)     {$S$};
\end{scope}

\begin{scope}[xshift=0cm,yshift=2cm]
   \node (2)  at (0,0)         {$\im (\sigma)$};
   \node (2')  at (-30:2cm)     {$ P_{0}$};
   \node (2'')  at (-30:4cm)     {$S$};
\end{scope}

\begin{scope}[xshift=0cm,yshift=4cm]
   \node (1)  at (0,0)         {$C$};
   \node (1')  at (-30:2cm)     {$C$};
   \node (1'')  at (-30:4cm)     {$$};
\end{scope}

\begin{scope}
   \node (a)  at  (210:4cm)     {$$};
   \node (a')  at (210:2cm)     {$T$};
   \node (a'')  at (0,0)        {$T$};
\end{scope}

\begin{scope}[xshift=0cm,yshift=2cm]
   \node (b)  at  (210:4cm)     {$K$};
   \node (b')  at (210:2cm)     {$P_{-1}$};
   \node (b'')  at (0,0)        {$\im (\sigma)$};
\end{scope}

\begin{scope}[xshift=0cm,yshift=4cm]
   \node (c)  at  (210:4cm)     {$K$};
   \node (c')  at (210:2cm)     {$C'$};
   \node (c'')  at (0,0)        {$C$};
\end{scope}

\draw[right hook -> , thin]  (1)  to  node  {$\mu$} (2);
\draw[right hook -> , thin]  (1')  to  node  {$\mu_0$} (2');

\draw[right hook -> , thin]  (b)  to  node  {$$} (b');
\draw[right hook -> , thin]  (c)  to  node  {$$} (c');

\draw[->> , thin]  (2)  to  node  {$q$} (3);
\draw[->> , thin]  (2')  to  node  {$q_0$} (3');
\draw[- , double]  (2'')  to  node  {$$} (3'');

\draw[->> , thin]  (c')  to  node  {$\c \pi_{\sigma}$} (c'');
\draw[->> , thin]  (b')  to  node  {$\pi_{\sigma}$} (b'');
\draw[- , double]  (a')  to  node  {$$} (a'');

\draw[- , double]  (1)  to  node  {$$} (1');
\draw[right hook -> , thin]  (2)  to  node  {$\iota_{\sigma}$} (2');
\draw[right hook -> , thin, below left]  (3)  to  node  {$(1:\c)\iota_{\sigma}$} (3');

\draw[- , double]  (c)  to  node  {$$} (b);
\draw[right hook -> , thin]  (c')  to  node  {$\mu _1 $} (b');
\draw[right hook -> , thin, below]  (c'')  to  node  {$$} (b'');

\draw[->> , thin]  (2')  to  node  {$$} (2'');
\draw[->> , thin, above]  (3')  to  node  {$\oplus$} (3'');
\draw[right hook -> , thin, above]  (3)  to  node  {$\oplus$} (3');

\draw[->> , thin]  (b')  to  node  {$q_1$} (a');
\draw[->> , thin, below]  (b'')  to  node  {$$} (a'');

\end{tikzpicture}
\]%
To see this, let (a) and (b) be satisfied. Since $\c S=0$, we can
pick $C=\c\im(\sigma)=\c P_{0}$. By (b), we have tht $(1:\c)\pi_{\sigma}$ is
an isomorphism. Hence, we can pick $T=(1:\c)\im(\sigma)=(1:\c)P_{-1}$,
and conclude that $(1:\c)\iota_{\sigma}$ is a monomorphism since
$(1:\c)\sigma$ is a monomorphism by (a) and $(1:\c)\pi_{\sigma}$
an isomorphism. The left part of the diagram is deduced from applying $\suc[\c][1][(1:\c)]$
to the morphism $\pi_{\sigma}$, and then applying the Snake Lemma
to the resulting diagram. 

In the case where the diagram above is given, then it is clear that (a) and (b) are satisfied. 
\end{rem}
\begin{rem}\label{rem. auxiliar}
Statement (c) is satisfied if and only if $\mathcal{D}_{\sigma}\cap\mathcal{F}_{I}\cap\mathcal{T}_{I}^{\bot_{1}}=0$
(see Proposition \ref{prop:recollement}(e) and Lemma \ref{lemita d sigma}).
\end{rem}


In addition to the previous equivalences, we have other alternatives
for statement (c) from Proposition \ref{prop:1} in Proposition \ref{prop:2}.
The following lemmas will be useful in the proof of Proposition \ref{prop:2}.


We will be using the following setting. 

\begin{setting}\label{setting 1y2} Consider Setting \ref{setting}
and also assume that statements (a) and (b) from Proposition \ref{prop:1}
are satisfied. Or, equivalently, that the diagram in Remark \ref{rem:presentacion diagrama}
exists. 

\end{setting}

We now have the following result.

\begin{lem}
\label{lem:T en D}With the conditions described in Setting \ref{setting 1y2}, we have that $\mathcal{T}_{I}\subseteq\mathcal{D}_{\sigma}$. 
\end{lem}

\begin{proof}
Consider the diagram in Remark \ref{rem:presentacion diagrama}. For
a module $M\in\mathcal{T}_{I}$ and a morphism $f:P_{-1}\rightarrow M$,
observe that $f\circ\mu_{1}=0$ since $C'\in\mathcal{C}_{I}$ and
$M\in\mathcal{T}_{I}$. Therefore, there is a morphism $f':T\rightarrow M$
such that $f'\circ q_{1}=f$. And hence, $f=f'\circ q\circ\pi_{\sigma}$.
Therefore, $M\in\mathcal{D}_{\pi_{\sigma}}$. Lastly, we note that
$M\in S^{\bot_{1}}$ by Setting \ref{setting}, and thus $M\in\mathcal{D}_{\pi_{\sigma}}\cap S^{\bot_{1}}=\mathcal{D}_{\sigma}$. 
\end{proof}
The following result is well-known, but we include it here for clarity and ease of reference.

\begin{lem}
\label{lem:epi vs pushout}\cite[Lemma 2.2]{breaz2018torsion} Consider
the commutative diagram below where the rows are exact. Then,
the middle square is a pushout if and only if $f$ is an epimorphism.
\[
\begin{tikzpicture}[-,>=to,shorten >=1pt,auto,node distance=2cm,main node/.style=,x=1cm,y=-1cm,scale=1.5]

\node (1) at (1,0) {$A$};
\node (2) at (2,0) {$B$};
\node (3) at (3,0) {$C$};
\node (4) at (4,0) {$D$};

\node (1') at (1,1) {$X$};
\node (2') at (2,1) {$ Y$};
\node (3') at (3,1) {$Z$};
\node (4') at (4,1) {$D$};

\draw[right hook -> , thin]  (1)  to  node  {$a$} (2);
\draw[right hook -> , thin]  (1')  to  node  {$x$} (2');

\draw[-> , thin]  (2)  to  node  {$b$} (3);
\draw[-> , thin]  (2')  to  node  {$y$} (3');

\draw[->> , thin]  (3)  to  node  {$$} (4);
\draw[->> , thin]  (3')  to  node  {$$} (4');

\draw[ -> , thin]  (1)  to  node  {$f$} (1');
\draw[-> , thin]  (2)  to  node  {$g$} (2');
\draw[-> , thin]  (3)  to  node  {$h$} (3');
\draw[- , double]  (4)  to  node  {$$} (4');

\end{tikzpicture}
\]%
\end{lem}

Now we state the following.

\begin{lem}
\label{lem:diagrama silting preenv}\label{lem:construccion}With the conditions described in  Setting \ref{setting 1y2} and the canonical short exact sequence
$\suc[IM][M][M/IM][\epsilon_{M}][\rho_{M}]$ for a given module $M\in\Rmod$.
Then, the following statements hold for $\alpha:=\Hom(P_{-1},M)$.
\begin{enumerate}
\item There is a commutative diagram with exact rows as the following: 
\[
\begin{tikzpicture}[-,>=to,shorten >=1pt,auto,node distance=2cm,main node/.style=,x=2cm,y=-2cm]

\node (1) at (1,0) {$K^{(\alpha)}$};
\node (2) at (2,0) {$P_{-1} ^{(\alpha)}$};
\node (3) at (3,0) {$P_{0} ^{(\alpha)}$};
\node (4) at (4,0) {$S ^{(\alpha)}$};

\node (1') at (1,1) {$L$};
\node (2') at (2,1) {$M$};
\node (3') at (3,1) {$T$};
\node (4') at (4,1) {$S^{(\alpha)}$};

\draw[right hook -> , thin]  (1)  to  node  {$i$} (2);
\draw[right hook -> , thin]  (1')  to  node  {$a$} (2');

\draw[-> , thin]  (2)  to  node  {$\sigma ^{(\alpha)}$} (3);
\draw[-> , thin]  (2')  to  node  {$\lambda$} (3');

\draw[->> , thin]  (3)  to  node  {$$} (4);
\draw[->> , thin]  (3')  to  node  {$$} (4');

\draw[ ->> , thin]  (1)  to  node  {$\nu$} (1');
\draw[-> , thin]  ( 2)  to  node  {$\phi$} (2');
\draw[-> , thin]  (3)  to  node  {$h$} (3');
\draw[- , double]  (4)  to  node  {$$} (4');

\end{tikzpicture}
\]%
where $\lambda$ is a $\mathcal{D}_{\sigma}$-preenvelope.
\item There is a morphism $q':T\rightarrow M/IM$ such that $q'\circ\lambda=\rho_{M}$. 
\item There is a morphism $\delta:P_{0}^{(\alpha)}\rightarrow M$ such that
$\rho_{M}\circ\delta=q'\circ h$. 
\item There is a morphism $\delta':P_{-1}^{(\alpha)}\rightarrow IM$ such
that $\epsilon_{M}\circ\delta'=\delta\circ\sigma^{(\alpha)}-\phi$.
\item If $M\in S^{\bot_{1}}$ and $\lambda$ is a monomorphism, then $M\in\mathcal{D}_{\sigma}$.
\item If $M\in\mathcal{D}_{\sigma}$, then $\lambda$ is a split monomorphism. 
\item If $S$ is silting with respect to $\sigma$, then the following statements
hold true.
\begin{enumerate}
\item There is a commutative diagram with exact rows as the following: 
\[
\begin{tikzpicture}[-,>=to,shorten >=1pt,auto,node distance=2cm,main node/.style=,x=2.5cm,y=-2cm]

\node (1) at (1,0) {$K^{(\alpha)}$};
\node (2) at (2,0) {$P_{-1} ^{(\alpha)}$};
\node (3) at (3,0) {$P_{0} ^{(\alpha)}$};
\node (4) at (4,0) {$S ^{(\alpha)}$};

\node (1') at (1,1) {$IM$};
\node (2') at (2,1) {$M$};
\node (3') at (3,1) {$M/IM \oplus S^{(\alpha)}$};
\node (4') at (4,1) {$S^{(\alpha)}$};

\draw[right hook -> , thin]  (1)  to  node  {$i$} (2);
\draw[right hook -> , thin]  (1')  to  node  {$\epsilon_{M}$} (2');

\draw[-> , thin]  (2)  to  node  {$\sigma ^{(\alpha)}$} (3);
\draw[-> , thin]  (2')  to  node  {$\left[\begin{smallmatrix}\rho_{M}\\0\end{smallmatrix}\right]$} (3');

\draw[->> , thin]  (3)  to  node  {$c$} (4);
\draw[->> , thin]  (3')  to  node  {$\left[\begin{smallmatrix}0 & 1\end{smallmatrix}\right]$} (4');

\draw[ ->> , thin]  (1)  to  node  {$\nu$} (1');
\draw[-> , thin]  ( 2)  to  node  {$\phi$} (2');
\draw[-> , thin]  (3)  to  node  {$\left[\begin{smallmatrix}h'\\c\end{smallmatrix}\right]$} (3');
\draw[- , double]  (4)  to  node  {$$} (4');

\end{tikzpicture}
\]%
where the morphism $\left[\begin{smallmatrix}\rho_{M}\\
0
\end{smallmatrix}\right]$ is a $\mathcal{D}_{\sigma}$-preenvelope. 
\item $\rho_{M}$ is a $\mathcal{D}_{\sigma}$-preenvelope.
\item $\delta'$ is an epimorphism. 
\end{enumerate}
\end{enumerate}
\end{lem}

\begin{proof}
\begin{enumerate}
\item The top row is obtained by taking the coproduct of $\alpha$ copies
of the exact sequence 
\[
\suc[K][P_{-1}\rightarrow P_{0}][S]\text{.}
\]
Let $\phi$ be the induced morphism $P_{-1}^{(\alpha)}\rightarrow M$,
$(x_{f})_{f\in\alpha}\mapsto\sum_{f\in\alpha}f(x_{f})$. Note that
$S^{(\alpha)}\in\Gen(S)=\mathcal{T}_{I}$. Hence, we can follow the
proof of \cite[Proposition 3.12]{siltingmodules} to conclude that
the pushout of $\sigma^{(\alpha)}$ and $\phi$ gives the desired
commutative diagram with exact rows where $\lambda$ is a $\mathcal{D}_{\sigma}$-preenvelope.
Also note that $\nu$ is an epimorphism by Lemma \ref{lem:epi vs pushout}. 
\item Since $M/IM\in\mathcal{T}_{I}\subseteq\mathcal{D}_{\sigma}$ by Lemma
\ref{lem:T en D} and $\lambda$ is a $\mathcal{D}_{\sigma}$-preenvelope
by (a), there is a morphism $q':T\rightarrow M/IM$ such that $q'\circ\lambda=\rho_{M}$. 
\item It follows from the facts that $\rho_{M}$ is an epimorphism and $P_{0}^{(\alpha)}$
is projective.
\item Note that $\rho_{M}\circ\delta\circ\sigma^{(\alpha)}=q'\circ h\circ\sigma^{(\alpha)}=q'\circ\lambda\circ\phi=\rho_{M}\circ\phi$.
Hence, by the universal property of kernels $\delta'$ exists. 
\item In this case, $M$ is a direct summand of a module which is in $\mathcal{D}_{\sigma}$
(see proof of the item (a)). Hence, $M\in\mathcal{D}_{\sigma}$. 
\item Let $M\in\mathcal{D}_{\sigma}$. Since $\lambda$ is a $\mathcal{D}_{\sigma}$-preenvelope,
there is a morphism $p:T\rightarrow M$ such that $p\circ\lambda=1$.
Therefore, $\lambda$ is a split monomorphism. 
\item [(g1)]Let $\lambda=\mu_{\lambda}\circ\pi_{\lambda}$ be the canonical
factorization of $\lambda$ with $\mu_{\lambda}:\im(\lambda)\rightarrow T$
and $\pi_{\lambda}:M\rightarrow\im(\lambda)$. Consider the equality
$q'\circ\lambda=\rho_{M}$. By the Snake Lemma, we have the following
commutative diagram. 
\[
\begin{tikzpicture}[-,>=to,shorten >=1pt,auto,node distance=2cm,main node/.style=,x=2cm,y=-2cm]

\node (1) at (1,0) {$L$};
\node (2) at (2,0) {$M$};
\node (3) at (3,0) {$\im (\lambda)$};

\node (1') at (1,1) {$IM$};
\node (2') at (2,1) {$M$};
\node (3') at (3,1) {$M/IM$};

\draw[right hook -> , thin]  (1)  to  node  {$a$} (2);
\draw[right hook -> , thin]  (1')  to  node  {$\epsilon_{M}$} (2');

\draw[->> , thin]  (2)  to  node  {$\pi_{\lambda}$} (3);
\draw[->> , thin]  (2')  to  node  {$\rho _M$} (3');


\draw[right hook -> , thin]  (1)  to  node  {$\nu '$} (1');
\draw[- , double]  ( 2)  to  node  {$$} (2');
\draw[->> , thin]  (3)  to  node  {$q' \circ \mu_{\lambda}$} (3');

\end{tikzpicture}
\]%
Here, note that $\nu'$ is an epimorphism and $q'\circ\mu_{\lambda}$
is a monomorphism. Indeed, $\Ker(q'\circ\mu_{\lambda})\in\mathcal{T}_{I}$
since it is a submodule of $T$, and $\Coker(\nu')\in\mathcal{C}_{I}$.
And hence, by the Snake Lemma, $\Coker(\nu')\cong\Ker(q'\circ\mu_{\lambda})\in\mathcal{C}_{I}\cap\mathcal{T}_{I}=0$.
Therefore, we can assume that $L=IM$, $\im(\lambda)=M/IM$, $\nu'=1$,
and $q'\circ\mu_{\lambda}=1$. Furthermore, this last equality tells
us that the short exact sequence 
\[
\suc[\im(\lambda)][T][S^{(\alpha)}]
\]
splits. Therefore, we can replace it with the short exact sequence
\[
M/IM\overset{\left[\begin{smallmatrix}1\\
0
\end{smallmatrix}\right]}{\hookrightarrow}M/IM\oplus S^{(\alpha)}\overset{\left[\begin{smallmatrix}0 & 1\end{smallmatrix}\right]}{\twoheadrightarrow}S^{(\alpha)}.
\]
\item [(g2)]It follows from the previous item. 
\item [(g3)]Observe that $\epsilon_{M}\circ\delta'\circ i=\delta\circ\sigma^{(\alpha)}\circ i-\phi\circ i=-\phi\circ i=-\epsilon_{M}\circ\nu$.
Therefore, $\delta'\circ i=-\nu$ and thus $\delta'$ is an epimorphism. 
\end{enumerate}
\end{proof}
\begin{prop}
\label{prop:2} With the conditions described in  Setting \ref{setting}, the following
statements are equivalent. 
\begin{enumerate}
\item [(1)]$S$ is silting with respect to $\sigma$. 
\item [(2)]Statements (a) and (b) from Proposition \ref{prop:1} hold true
together with:
\begin{enumerate}
\item [(c')] The canonical epimorphism $R\rightarrow R/I$ is a $\mathcal{D}_{\sigma}$-preenvelope.
\end{enumerate}
\item [(3)] Statements (a) and (b) from Proposition \ref{prop:1} hold
true together with:
\begin{enumerate}
\item [(c'')] $\mathcal{D}_{\sigma}\cap\mathcal{F}_{I}=0$. 
\end{enumerate}
\item [(4)] Statements (a) and (b) from Proposition \ref{prop:1} hold
true together with:
\begin{enumerate}
\item [(c''')] $\mathcal{D}_{\sigma}\subseteq\mathcal{T}_{I}$. 
\end{enumerate}
\end{enumerate}
\end{prop}

\begin{proof} $(1)\implies (2)$ Follows from Lemma \ref{lem:diagrama silting preenv}(g2)
and Proposition \ref{prop:1}. 

$(2)\implies (3)$ Let $M\in\mathcal{D}_{\sigma}\cap\mathcal{F}$.
For every $m\in M$, consider the morphism $f_{m}:R\rightarrow M$
such that $r\mapsto r\cdot m$. Since $M\in\mathcal{D}_{\sigma}$
and $\rho_{R}:R\rightarrow R/I$ is a $\mathcal{D}_{\sigma}$-preenvelope,
there is a morphism $g_{m}:R/I\rightarrow M$ such that $g_{m}\circ\rho_{R}=f_{m}$.
But, since $M\in\mathcal{F}_{I}$ and $R/I\in\mathcal{T}_{I}$, $g_{m}=0$
for all $m\in M$. Therefore, $M=0$. 

$(3)\implies (1)$ It is enough to prove statement (c) from Proposition
\ref{prop:1}. But this follows straightforward by Remark \ref{rem. auxiliar}.

Lastly, equivalence $(3)\iff (4)$ follows from routine
arguments. 
\end{proof}

\section{The trace ideal of a projective module and its TTF triple}

Let $P\in\Proj(R)$ and consider its trace $\Tr(P)=\sum_{f\in P^{*}}\im(f)$, a two-sided idempotent ideal. Hence, there is a corresponding TTF triple associated to it as follows:
\[
\triple[\Gen(I)][\Gen(R/I)][(R/I)^{\bot_{0}}]=\triple=(\Gen(P),P^{\bot_{0}},(P^{\bot_{0}})^{\bot_{0}}).
\]
 For more details, the reader is referred to \cite{zbMATH03533104,zbMATH03286961}.
\begin{example}
Let $I$ be an idempotent two-sided ideal and $\triple$ be the corresponding
TTF triple. Observe that, in case there is an epimorphism $P\rightarrow I$
with $P\in\Proj(R)$, then $I$ is the trace of $P$ if and only if $P\in\mathcal{C}_{I}$.
Using this fact, we have the following examples. 
\begin{enumerate}
\item If $R$ is a perfect ring, then $I$ is the trace of a projective
module. Indeed, from \cite[Lemma 7.3]{zbMATH06669084}, we known  that, if $P\rightarrow I$ is the projective cover of $I$, then $P\in\mathcal{C}_{I}$. Therefore, $I=\Tr(P)$. 
\item If $R$ is a hereditary ring (i.e. a  ring where the class of projective modules is closed under submodules), then $I$ is projective. Thus, $I$
is the trace of the projective module $I$. 
\item Recall that $R$ is a von Neumann regular ring if, for every principal
left ideal $Rx$, there is an idempotent $e\in R$ such that $Rx=Re$.
Note that in this case $P:=\bigoplus_{x\in I}Rx\in\mathcal{C}_{I}\cap\Proj(R)$.
Indeed, for every $x\in I$ there is an idempotent $e\in I$ such
that $Rx=Re$. Moreover, on the one hand, $Re\subseteq Ie$ since
\[
Re\subseteq I=Ie\oplus I(1-e)
\]
 and $Re\cap I(1-e)=0$. And, on the other hand, $Ie\subseteq Re$
since $I\subseteq R$. Therefore, for every $x\in I$, $Rx=Re=Ie$. And thus, $Rx\in\mathcal{C}_{I}$ because
$Ie$ is a direct summand of $I$, and $Rx$ is projective since $Re$ is a direct summand of $R$. And thus, $I$
is the trace of $P$. For further details on this example, the reader
is referred to \cite[Theorem 4.5]{bravo2019t}.
\item Recently Li, \cite{zbMATH07799812}, showed different situations
where every $M\in\mathcal{C}_{I}$ admits an epimorphism $P\rightarrow M$
with $P\in\Proj(R)\cap\mathcal{C}_{I}$.
\end{enumerate}
\end{example}

\begin{prop}
\label{prop:a y b vs traza}Consider the conditions described in Setting \ref{setting 1y2}
for $S:=R/I$. If $R/I$ admits a projective cover, then $I$ is the
trace of a projective module.
\end{prop}

\begin{proof}
Let $\pi':Q\rightarrow R/I$ be a projective cover. Consider the following
short exact sequences 
\begin{align}
\suc[I][R][R/I][\epsilon_{R}][\rho_{R}] \\
\suc[K][Q][R/I][k][\pi']\\
\suc[\operatorname{Im}(\sigma)][P_{0}][R/I][\iota_{\sigma}][c] 
\end{align}
On one hand, by Lemma \ref{lem:cubierta proyectiva}, we have
that there is $P\in\Proj(R)$ such that $I\cong K\oplus P$, and thus
$P,K\in\mathcal{C}_{I}$. In particular, by observing the exact sequence
(2), we can conclude that $\c Q\cong\c K=K$. Also by
Lemma \ref{lem:cubierta proyectiva}, we have that there is $P'\in\Proj(R)$
such that $\im(\sigma)\cong K\oplus P'$. Observe that, since $P'\in\Proj(R)$,
the epimorphism given by the composition of the following maps
\[
P_{-1}\overset{\pi_{\sigma}}{\rightarrow}\im(\sigma)=K\oplus P'\overset{\left[\begin{smallmatrix}0 & 1\end{smallmatrix}\right]}{\rightarrow}P'
\]
splits. Thus, we have that there is $\hat{Q}\in\Proj(R)$ such
that $P_{-1}=\hat{Q}\oplus P'$, and therefore $\pi_{\sigma}:\hat{Q}\oplus P'\rightarrow K\oplus P'$
can be expressed as some matrix $\left[\begin{smallmatrix}a & b\\
0 & 1
\end{smallmatrix}\right]$. 

Now, by statement (b) from Proposition \ref{prop:1}, 
\[
(1:\c)(\pi_{\sigma}):(1:\c)P_{-1}\rightarrow(1:\c)\im\sigma
\]
 is an isomorphism. Observe that $(1:\c)P_{-1}=(1:\c)\hat{Q}\oplus(1:\c)P'$
and that $(1:\c)\im\sigma=(1:\c)P'$, since $(1:\c)K=0$. Thus, $(1:\c)(\pi_{\sigma})$
can be expressed as the matrix 
\[
\left[\begin{smallmatrix}(1:\c)(b)\\
(1:\c)(1)
\end{smallmatrix}\right]=\left[\begin{smallmatrix}(1:\c)(b)\\
1
\end{smallmatrix}\right]:(1:\c)\hat{Q}\oplus(1:\c)P'\rightarrow(1:\c)P'.
\]
Note that, if this matrix is an isomorphism, then $(1:\c)\hat{Q}=0$,
and thus $\hat{Q}\in\mathcal{C}_{I}$.

Let us prove that $a:\hat{Q}\rightarrow K$ is an epimorphism. For
this, recall that $\c\pi_{\sigma}=\left[\begin{smallmatrix}\c a & \c b\\
0 & 1
\end{smallmatrix}\right]$ is an epimorphism by Remark \ref{rem:presentacion diagrama}.
Thus, it can be seen that $\c a$ also is an epimorphism. But
$\c a=a$ since $(1:\c)\hat{Q}=0$. Therefore, $a$ is an epimorphism. 

In conclusion, there is an epimorphism $\hat{Q}\oplus P\overset{\left[\begin{smallmatrix}\c a & 0\\
0 & 1
\end{smallmatrix}\right]}{\rightarrow}K\oplus P=I$ with $\hat{Q}\oplus P\in\mathcal{C}_{I}\cap\Proj(R)$. Therefore,
$I$ is the trace of $\hat{Q}\oplus P$. 
\end{proof}
\begin{example}
Let $R$ be a local (not necessarily commutative) ring with a nonzero
idempotent maximal ideal $I$.  If we assume that $R/I$ is silting,
then it follows from Propositions \ref{prop:1} and \ref{prop:a y b vs traza}
that $I$ is the trace of a projective module since the natural projection $R\rightarrow R/I$ is a projective cover. 
Next,
it is known that all projective modules over a local ring are free
(see \cite[Corollary 3.3]{zbMATH03187694}). Hence, we get that $I=R$,
a contradiction. Therefore, $R/I$ is not a silting module. For an
alternative approach to this example, when $R$ a commutative local
ring, see \cite[Example 5.4]{silting2017}.
\end{example}

We now  prove Theorem B.
\begin{thmB}Let $I$ be an idempotent two-sided
ideal of $R$. If $R/I$ admits a projective cover and $R/I$ is silting,
then $I$ is the trace of a projective module. 

\end{thmB}
\begin{proof}
 Since $R/I$ is silting, it follows from Proposition
\ref{prop:1} that $R/I$ satisfies Setting \ref{setting 1y2}. Therefore,
$I$ is the trace of a projective module by Proposition \ref{prop:a y b vs traza}. 
\end{proof}
Consider now the following set of conditions.
\begin{setting}\label{setting2} 
\begin{enumerate}
    \item Let $P$ be a projective module,
    \item $P^{*}:=\Hom(P,R)$.
    \item $I:=Tr(P)$, and $\triple$ the corresponding TTF triple.
    \item Consider $P_{-1}:=P^{(P^{*})}$, $P_{0}:=R$, the canonical epimorphism $\pi_{\sigma}:P_{-1}\rightarrow I$, the inclusion $\iota_{\sigma}:I\rightarrow P_{0}$, and the morphism $\sigma:=\iota_{\sigma}\circ\pi_{\sigma}$.
\end{enumerate}
\end{setting}

\begin{rem}
Note that Setting \ref{setting2} is a special case of Setting \ref{setting 1y2}.
Indeed, statement (a) from Proposition \ref{prop:1} is satisfied
since $(1:\c)P_{0}=R/I$ and $(1:\c)P_{-1}=0$. And, statement (b)
from Proposition \ref{prop:1} is satisfied since $(1:\c)\im\sigma=(1:\c)I=0$
and $(1:\c)P_{-1}=0$.
\end{rem}

The following theorem is inspired by \cite[Proposition  2.1]{breaz2018torsion}. 
\begin{thm}
\label{thm:traza de proyectivo}Consider the conditions described in Setting \ref{setting2}.
Then, $R/I$ is a silting module with respect to $\rho$, where $\rho$
is the morphism 
\[
\left[\begin{smallmatrix}\sigma & 0 & 0\\
0 & 0 & 1
\end{smallmatrix}\right]:P\oplus P\oplus R\rightarrow R\oplus R.
\]
\end{thm}

\begin{proof}
Consider the morphisms $\sigma:P\rightarrow R$ and $\left[\begin{smallmatrix}0 & 1\end{smallmatrix}\right]:P\oplus R\rightarrow R$.
Note that $\sigma\oplus\left[\begin{smallmatrix}0 & 1\end{smallmatrix}\right]=\rho$,
and thus $\Coker(\rho)=\Coker(\sigma)=R/I$. 

Let us prove that $\mathcal{T}_{I}=\mathcal{D}_{\rho}$.
On the one hand, if $M\in\mathcal{T}_{I}$, then every $f\in\Hom(P\oplus P\oplus R,M)$
can be seen as the matrix $\left[\begin{smallmatrix}0 & 0 & a\end{smallmatrix}\right]$
since $P\in\mathcal{C}_{I}$. And hence, $M\in\mathcal{D}_{\rho}$
because 
\[
\left[\begin{smallmatrix}0 & 0 & a\end{smallmatrix}\right]=\left[\begin{smallmatrix}0 & a\end{smallmatrix}\right]\circ\left[\begin{smallmatrix}\sigma & 0 & 0\\
0 & 0 & 1
\end{smallmatrix}\right].
\]
On the other hand, for every $b\in\Hom(P,M)$, we can consider the
morphism $\left[\begin{smallmatrix}0 & b & 0\end{smallmatrix}\right]:P\oplus P\oplus R\rightarrow M$.
If $M\in\mathcal{D}_{\rho}$, then there is a morphism $\left[\begin{smallmatrix}f & g\end{smallmatrix}\right]:R\oplus R\rightarrow M$
such that 
\[
\left[\begin{smallmatrix}0 & b & 0\end{smallmatrix}\right]=\left[\begin{smallmatrix}f & g\end{smallmatrix}\right]\circ\left[\begin{smallmatrix}\sigma & 0 & 0\\
0 & 0 & 1
\end{smallmatrix}\right]=\left[\begin{smallmatrix}f\circ\sigma & 0 & g\end{smallmatrix}\right].
\]
Therefore, $M\in P^{\bot_{0}}=\mathcal{T}_{I}$. 
\end{proof}

Note that Theorem A is contained in Theorem \ref{thm:traza de proyectivo}.
\begin{thmA} Let $I$ be an idempotent two-sided
ideal of $R$. If $I$ is the trace of a projective left (resp. right) module,
then $R/I$ is a silting left (resp. right) module. 
\end{thmA}

\begin{proof} Indeed, we have that Setting \ref{setting2} is fulfilled. Therefore,  $R/I$ is silting with respect to the morphism built in Theorem \ref{thm:traza de proyectivo}.
\end{proof}

\begin{example}
Let $R$ be the ring of continuous functions on $[0,1]$ and $I$
be the ideal of all functions vanishing in some neighborhood of $0$.
By 
 \cite[Example (2)]{zbMATH03187694}, $I$ is a projective module such that $I=Tr (I)$. Therefore, it follows from
Theorem \ref{thm:traza de proyectivo} that $R/I$ is a silting module
and $\mathcal{T}_{I}$ a silting class. 
\end{example}

\begin{cor}
Let $I$ be an idempotent two-sided ideal that is finitely generated
as a right ideal. Then, $R/I$ is a silting left module. 
\end{cor}

\begin{proof}
It was proved in \cite[Corollary 2.7]{zbMATH03699117} that in this
case $I$ is the trace of a projective left module $P$. Therefore,
$R/I$ is a silting module by Theorem \ref{thm:traza de proyectivo}. 
\end{proof}
\begin{thm}
\label{thm:cubierta vs silting}Consider the conditions described in Setting \ref{setting 1y2}
for $S:=R/I$. If $R/I$ admits a projective cover, then $R/I$ is
a silting module. 
\end{thm}

\begin{proof}
It follows from Proposition \ref{prop:a y b vs traza} and Theorem
\ref{thm:traza de proyectivo}.
\end{proof}
\begin{cor}
Consider the conditions described in Setting \ref{setting 1y2} for $S:=R/I$. If $I$ is
an idempotent ideal such that $I\subseteq J$, where $J$ is the Jacobson
radical of $R$, then $R/I$ is a silting module. 
\end{cor}

\begin{proof}
It is known that, in this case, the natural projection $\rho_{R}:R\rightarrow R/I$
is a projective cover (see \cite[Lemma 2.8.36]{zbMATH00193091}).
And thus, the result follows from Theorem \ref{thm:cubierta vs silting}.
\end{proof}
%

A tCG torsion pair in $R$-Mod, is a torsion pair on which the associated Happel-Reiten-Smal\o \ t-structure is compactly generated. For more details, we refer the reader to \cite{bravo2019t}.


\begin{cor}
\label{cor:semiperfectos}Let $R$ be a semiperfect ring and $I$
an idempotent two-sided ideal. Then, the following statements are
equivalent. 
\begin{enumerate}
\item $R/I$ is a silting module. 
\item $I$ is the trace of a projective module. 
\item There is a set $\{M_{i}\}_{i\in X}$ of finitely generated projective
modules in $\mathcal{C}_{I}$ such that $\Gen(\bigoplus_{i\in X}M_{i})=\mathcal{C}_{I}$. 
\item There is a set $\{M_{i}\}_{i\in X}$ of finitely generated modules
in $\mathcal{C}_{I}$ such that $\Gen(\bigoplus_{i\in X}M_{i})=\mathcal{C}_{I}$. 
\item $(\mathcal{C}_{I},\mathcal{T}_{I})$ is a tCG torsion pair. 
\end{enumerate}
\end{cor}

\begin{proof}
Let us assume that $I\neq0$.

$(a)\Longleftrightarrow (b)$ Follows from Theorems A and B. 

$(b)\implies (c)$ Recall that there is a set of idempotents $e_{1},\cdots,e_{n}\in R$
such that every $P\in\Proj(R)$ is isomorphic to a module of the form
\[
Re_{1}^{(A_{1})}\oplus Re_{2}^{(A_{2})}\oplus\cdots\oplus Re_{n}^{(A_{n})}
\]
 \cite[Section 27]{Anderson-Fuller}. Hence, if $I$ is the trace
of a projective module $P$, then there is a set $\{f_{i}\}_{i=1}^{k}\subseteq\{e_{i}\}_{i=1}^{n}$
such that $Rf_{i}\in\mathcal{C}_{I}$ for all $1\leq i\leq k$. Indeed,
since $P\in\mathcal{C}_{I}$, all of its indecomposable summands belong
to $\mathcal{C}_{I}$. Lastly, we note that $\Gen(\bigoplus_{i=1}^{k}Rf_{i})=\Gen(I)=\mathcal{C}_{I}$. 

$(c)\implies (d)$ It is trivial. 

$(d)\implies (e)$ Note that $\mathcal{T}_{I}=\underrightarrow{\lim}\mathcal{T}_{I}$
since $\mathcal{T}_{I}$ is closed under quotients. Therefore, by
\cite[Corollary 3.9]{bravo2019t}, it is enough to prove that $\mathcal{T}_{I}=\bigcap_{i\in X}M_{i}^{\bot_{0}}$.
This equality follows as a consequence of $(d)$. 

$(e)\implies (b)$ Note that $\mathcal{T}_{I}=\underrightarrow{\lim}\mathcal{T}_{I}$
since $\mathcal{T}_{I}$ is closed under quotients.Therefore, by \cite[Corollary 3.9]{bravo2019t},
there is a set $\{M_{i}\}_{i\in X}$ of finitely generated modules
in $\mathcal{C}_{I}$ such that $\mathcal{T}_{I}=\bigcap_{i\in X}M_{i}^{\bot_{0}}$.
Let $P_{i}\rightarrow M_{i}$ be the projective cover of $M_{i}$.
By \cite[Lemma 7.3]{zbMATH06669084}, we know that $P_{i}\in\mathcal{C}_{I}$
for all $i\in X$. Observe that $P_{i}$ also is a finitely generated
module and that $\mathcal{T}_{I}=(\bigoplus_{i\in X}M_{i})^{\bot_{0}}=(\bigoplus_{i\in X}P_{i}$)$^{\bot_{0}}$.
We claim that $I=Tr(P)$ where $P=\bigoplus_{i\in X}P_{i}$. To see this,
consider the natural short exact sequence 
\[
\suc[Tr(P)][I][C]
\]
Next, observe that, $C\in\Gen(I)=\mathcal{C}_{I}$. Since $P\in\Proj(R)$, then 
$C\in P^{\bot_{0}}=\mathcal{T}_{I}$. Therefore, $C=0$ and thus
$I=Tr(P)$. 
\end{proof}
\begin{example}
\label{exa:semiperfecto,local} Let $R$ be a semiperfect ring and
$J$ be the Jacobson radical of $R$. We claim that $R/J$ is not
a silting module whenever $J$ is an idempotent nonzero ideal. Indeed,
if $R/J$ is a silting module, then there is a set $\{M_{i}\}_{i\in X}$
of finitely generated projective modules in $\mathcal{C}_{J}$ such
that $\Gen(\bigoplus_{i\in X}M_{i})=\mathcal{C}_{J}$ by Corollary
\ref{cor:semiperfectos}. Since $J\neq0$ and $J\in\mathcal{C}_{J}$,
there is $i\in X$ such that $M_{i}\neq0$. Note that $M_{i}=JM_{i}$
since $M_{i}\in\mathcal{C}_{J}$. But this contradicts Nakayama Lemma
\cite[Proposition 2.5.24]{zbMATH00193091}. Therefore, $R/J$ is not
a silting module. 
\end{example}

Let $I$ be an idempotent two-sided ideal and $\mathcal{H}_{I}$ be
the heart of the Happel-Reiten-Smal{\o} $t$-structure associated
to $(\mathcal{T}_{I},\mathcal{F}_{I})$. If $\mathcal{T}_{I}$ is
a silting class, then $\mathcal{H}_{I}$ has a projective generator.
However, it is not clear if there is a projective generator in $\mathcal{H}_{I}$
with zero cohomology equal to $R/I$. Therefore, the following question
arises.

\begin{question} Let $I$ be an idempotent two-sided ideal of $R$
and $\mathcal{T}_{I}=\Gen(R/I)$ be a silting class. Is $R/I$ a silting
module? 

\end{question}

 \footnotesize

\vskip3mm 
\noindent Alejandro Argud\'in-Monroy\\
Departamento de Matem\'aticas, Facultad de Ciencias,\\
Universidad Nacional Aut\'onoma de M\'exico,\\ 
Circuito Exterior, Ciudad Universitaria,\\
CDMX 04510, M\'EXICO.\\
{\tt argudin@ciencias.unam.mx}

\vskip3mm 
\noindent Daniel Bravo\\
Instituto de Ciencias F\'isicas y Matem\'aticas\\
Edificio Emilio Pugin, Campus Isla Teja\\
Universidad Austral de Chile\\
5090000 Valdivia, CHILE\\ 
{\tt daniel.bravo@uach.cl}

\vskip3mm 
\noindent Carlos E. Parra\\
Instituto de Ciencias F\'isicas y Matem\'aticas\\
Edificio Emilio Pugin, Campus Isla Teja\\
Universidad Austral de Chile\\
5090000 Valdivia, CHILE\\
{\tt carlos.parra@uach.cl}
\end{document}